\documentclass[12pt,leqno]{article}
\usepackage{latexsym,amsfonts,epsfig}
\usepackage{amsmath,amssymb,amsthm}
\usepackage[notref, notcite]{}
\usepackage[colorlinks,linkcolor=red,citecolor=blue,urlcolor=blue]{hyperref}
\usepackage[margin=0.5in]{geometry}
\usepackage{graphicx}
\input epsf
\newtheorem{theorem}{Theorem}[section]

\newtheorem{remark}{Remark}[section]

\RequirePackage[colorlinks,citecolor=blue,urlcolor=blue]{hyperref}

\begin{document}

\title{A Markovian Gau\ss \, Inequality for Asymmetric Deviations from the Mode of Symmetric Unimodal Distributions}

\author{ Chris A.J. Klaassen \\
Korteweg-de Vries Insitute for Mathematics \\
University of Amsterdam}

\maketitle

\date{}



\begin{abstract}
For a random variable with a unimodal distribution and finite second moment Gau\ss \, (1823) proved a sharp bound on the probability of the random variable to be outside a symmetric interval around its mode.
An alternative proof for it is given based on Khintchine's representation of unimodal random variables.
Analogously, a sharp inequality is proved for the slightly broader class of unimodal distributions with finite first absolute moment, which might be called a Markovian Gau\ss\ inequality.
For symmetric unimodal distributions with finite second moment Semenikhin (2019) generalized Gau\ss's inequality to arbitrary intervals.
For the class of symmetric unimodal distributions with finite first absolute moment we construct a Markovian version of it.
Related inequalities of Volkov (1969) and Sellke and Sellke (1997) will be discussed as well.
\end{abstract}

\maketitle
\section{Introduction}\label{introduction}
On February 15, 1821 Johann Carl Friedrich Gau\ss \, presented the first part of his {\em Theoria Combinationis Observationum Erroribus Minimis Obnoxiae} to the Royal Academy in G\"ottingen.
It was published in \cite{Gaus} and translated from Latin into English in \cite{Stew}.
It contains in its Sections 9 and 10 on pages 8--12 results that imply that for any unimodal random variable $X$ with mode at $\nu$ the inequality
\begin{equation}\label{Gauss1}
P(|X - \nu| \geq v) \leq \frac 49\, \frac {E((X - \nu)^2)}{v^2},\quad v>0,
\end{equation}
holds.
The well known Bienaym\'e-Chebyshev inequality, which holds for arbitrary random variables $X$ with mean $\mu,$ is given by
\begin{equation}\label{BC}
 P(|X - \mu| \geq v) \leq \frac{E((X - \mu)^2)}{v^2},\quad v>0,
\end{equation}
and was first published by Bienaym\'e \cite{Bien} in 1853, but not in an obvious form, cf. \cite{Heyd}.
It was also published by Chebyshev \cite{Cheb} in 1867, who used the inequality to prove a law of large numbers.

Interestingly, Gau\ss's inequality precedes the Bienaym\'e-Chebyshev inequality, by three decades even.
A detailed description of Gau\ss's proof is presented by \cite{Hoog}.
\cite{Sell} describes the history of the Gau\ss \, inequality and refers to \cite{Puke} for three proofs, namely Gau\ss's one, a variation on it, and the one from exercise 4 on page 256 of Cram\'er \cite{Cram}.
An elegant fourth proof was given by Volkov \cite{Volk} via a basic inequality.
We will discuss this proof in the next section, where we will also present a fifth proof based on the Khintchine characterization of unimodal distributions.
This Khintchine approach was applied also in \cite{Ion}.
A slight adaptation of this fifth proof yields
\begin{equation}\label{GaussM1}
P(|X - \nu| \geq v) \leq \frac {E(|X - \nu|)}{2v},\quad v>0,
\end{equation}
which may be viewed as a Markov version of Gau\ss's inequality.
This inequality is also presented in Section \ref{Gi}.

A further generalization consists in replacing $(X-\nu)^2$ in (\ref{Gauss1}) and $|X-\nu|$ in (\ref{GaussM1}) by another even function of $X-\nu$.
Such a generalization was presented by Sellke and Sellke in \cite{Sell}; we present a simple reformulation of it in the next section together with a proof via Volkov's inequality.

Restricting attention to symmetric unimodal distributions Semenikhin \cite{Seme} generalized the Gau\ss\ inequality from symmetric to possibly asymmetric intervals around the mode.
He mentions several application areas for this type of inequalities.
These inequalities can be used also for control charts in statistical process control.
We discuss and reformulate Semenikhin's result in Section \ref{Gig} and we present a Markov version of it there.
The proof of this Markov version is also based on our approach via Khintchine's characterization.

\section{Gau\ss's inequality}\label{Gi}

 As mentioned in the Introduction Gau\ss's inequality may be proved via an inequality of Volkov. As his publication \cite{Volk} in Russian is hard to find, we present a version of his inequality and proof here.

\begin{theorem}[Volkov's inequality]\label{th:Volkov}
Let $X$ be a random variable on $[0, \infty)$ with nonincreasing density $f(\cdot)$ and let $\psi(\cdot)$ be an absolutely continuous function on $[0,\infty)$ with Radon-Nikodym derivative $\psi'(\cdot)$ with respect to Lebesgue measure. If $\psi(\cdot)$ is nondecreasing and
\begin{equation}\label{conditionong}
\psi(x) \geq x,\quad x \geq 0,
\end{equation}
holds, then we have
\begin{equation}\label{Volkovinequality}
P( X \geq \psi(0)) \leq E\left(\psi'(X)\right).
\end{equation}
\end{theorem}
\begin{proof}[Proof of Theorem \ref{th:Volkov}]
By the substitution theorem for integrals and the properties of $\psi(\cdot)$ we have
\begin{equation}\label{substitution}
E\left(\psi'(X)\right) = \int_0^\infty \psi'(x) f(x)\, dx = \int_0^\infty f(x)\, d\psi(x) = \int_{\psi(0)}^\infty f\left(\psi^{-1}(y)\right)\, dy.
\end{equation}
As (\ref{conditionong}) implies $\psi^{-1}(y) \leq y$ for $y\geq \psi(0),$ this yields (\ref{Volkovinequality}) in view of the monotonicity of $f(\cdot).$
\end{proof}

In stead of showing how Gau\ss's inequality follows from this inequality, as done in \cite{Volk}, we present a simplified version of the generalization of Gau\ss's inequality by \cite{Sell}, which we prove via Volkov's inequality.

\begin{theorem}[Inequality of Sellke and Sellke]\label{th:2Sellke}
For $v>0$ let $g(\cdot)$ be a nonnegative, nondecreasing function on $[0,\infty)$ that is positive at some point in $(0,v).$ With
\begin{equation}\label{root}
x_v = \sup \left\{ x > v \ :\ \int_0^x g(y)\, dy - g(x)(x-v) \geq 0 \right\}
\end{equation}
and for any unimodal random variable $X$ with mode at $\nu,$ the inequality
\begin{equation}\label{Sellkeinequality}
P(|X - \nu| \geq v) \leq  \frac{Eg(|X - \nu|)}{g(x_v-)}
\end{equation}
holds with equality if $x_v$ is finite, $|X - \nu|$ has a uniform distribution on $(0, x_v),$ and $g(\cdot)$ is continuous at $x_v.$
\end{theorem}
\begin{remark}
Applying this theorem with $g(x)=x^2$ one immediately obtains the Gau\ss\ inequality (\ref{Gauss1}).
Similarly $g(x)=x$ results in its Markov version (\ref{GaussM1}).
\end{remark}

\noindent
\begin{proof}[Proof of Theorem \ref{th:2Sellke}]
As $g(\cdot)$ is nonnegative, the function $\chi(x) = \int_0^x g(y)\, dy - g(x)(x-v)$ is nonnegative for $x \in[0,v].$
Furthermore, $\chi(v)$ is positive and finite, and for $v \leq x < y$ we have
\begin{eqnarray}\label{chi}
\lefteqn{\chi(x) - \chi(y) = -\int_x^y g(z)\, dz - g(x)(x-v) + g(y)(y-v) }\\
&& \geq -(y-x)g(y) - g(x)(x-v) + g(y)(y-v) = (g(y) - g(x))(x-v) \geq 0, \nonumber
\end{eqnarray}
which shows that $\chi(\cdot)$ is nonincreasing on $[v,\infty).$ It follows that for $\epsilon > 0$ sufficiently small $\chi(x_v - \epsilon) \geq 0$ holds. Note that $x_v$ might be infinity, which in view of (\ref{chi}) can happen only if $g(\cdot)$ is bounded.

 We define
\begin{equation}\label{psig}
\psi(x) =  \frac 1{g(x_v - \epsilon)} \int_0^x g(y)\, dy + v,\quad x \geq 0,
\end{equation}
and note that this function is nondecreasing, absolutely continuous with derivative \\ $g(x)/g(x_v -\epsilon),$ and hence convex. It follows that $\psi(x) - x$ is also convex and attains its minimum at $x_v - \epsilon$ with value $\psi(x_v - \epsilon) - (x_v - \epsilon) = \chi(x_v - \epsilon)/g(x_v - \epsilon) \geq  0.$ We conclude that $\psi(\cdot)$ satisfies the conditions of Theorem \ref{th:Volkov}.

Note that $|X - \nu|$ has a unimodal distribution on $[0,\infty)$ with a nonincreasing density on $(0,\infty)$ and a possible point mass, $1-p$ say, at 0. We apply Volkov's inequality (\ref{Volkovinequality}) to $|X - \nu|$ conditionally, given $|X - \nu|>0,$ with $\psi(\cdot)$ from (\ref{psig}) and obtain for $v>0$
\begin{eqnarray}\label{Sellke1}
\lefteqn{P(|X - \nu| \geq v) = p\, P(|X - \nu| \geq v \mid |X - \nu|>0) \nonumber } \\
&& \leq \frac p{g(x_v - \epsilon)} E (g(|X - \nu|) \, \mid \, |X - \nu|>0) \leq \frac {E(g(|X - \nu|))}{g(x_v - \epsilon)}.
\end{eqnarray}
Taking the limit for $\epsilon \downarrow 0$ we arrive at (\ref{Sellkeinequality}).

If $g(\cdot)$ is continuous at $x_v$ and $|X - \nu|$ has a uniform distribution on $(0, x_v),$ then  $P(|X - \nu| \geq v) = 1 - v/x_v$ holds as well as $E(g(|X - \nu|)) = x_v^{-1} \int_0^{x_v} g(y)\,dy \\ = g(x_v)(1 - v/x_v)$ in view of $\chi(x_v) =0.$
\end{proof}

\begin{remark}\label{symmetricSellke}
Equality in (\ref{Sellkeinequality}) is also attained if $X - \nu$ is uniform on $(-x_v,x_v)$. This means that Theorem \ref{th:2Sellke} yields sharp bounds also if the class of unimodal distributions is restricted to distributions symmetric about their mode.
\end{remark}

\begin{remark}\label{sameinequality}
It may be verified that the constant $A$ from \cite{Sell} (see also its Corollary 2) equals $g(x_v-)$ from (\ref{Sellkeinequality}).
\end{remark}

The complete Gau\ss\ inequality is implied by Theorem \ref{th:2Sellke}, as we show next.

\begin{theorem}[Gau\ss's inequality]\label{th:Gauss}
For any unimodal random variable $X$ with mode at $\nu,$ the inequality
\begin{eqnarray}\label{Gaussinequality}
\qquad P(|X - \nu| \geq v)\leq  {\left\{\begin{array}{rcl}
\frac 49 \, \frac {E((X - \nu)^2)}{v^2},& & v \geq \frac {2\sqrt{E((X - \nu)^2)}}{\sqrt 3}, \\
\\
1- \frac v{\sqrt{3 E((X - \nu)^2)}}, &  & 0\leq v < \frac {2\sqrt{E((X - \nu)^2)}}{\sqrt 3},\\
\end{array}\right.}
\end{eqnarray}
holds with equality if the distribution of $|X - \nu|$ is the mixture of a uniform distribution on $(0,c)$ with
\begin{equation}\label{uniform}
c = \left(\tfrac 32 v \right) \vee \sqrt{3 E((X - \nu)^2)}
\end{equation}
and a distribution degenerate at $0$ such that the point mass at $0$ equals
\begin{equation}\label{pointmass}
\left(1- \tfrac{4E((X - \nu)^2)}{3v^2} \right) \vee 0.
\end{equation}
\end{theorem}
\begin{proof}[Proof of Theorem \ref{th:Gauss}]
We apply the inequality of Sellke and Sellke with $g(x)=x^2$ and obtain $x_v=3v/2$ and the first inequality in (\ref{Gaussinequality}).
Applying Theorem \ref{th:2Sellke} again with
\begin{equation}\label{defg}
g(x) = x^2 + 3 E((X-\nu)^2)\left[\frac {2\sqrt{E((X-\nu)^2)}}{{\sqrt 3}\ v} -1 \right]
\end{equation}
 we obtain $x_v = \sqrt{3E((X-\nu)^2)}$ and the second inequality in (\ref{Gaussinequality}).

Straightforward computation shows that the obtained inequalities (\ref{Gaussinequality}) are sharp.
\end{proof}

\begin{remark}\label{optimalityGauss}
As Gau\ss \, restricted attention to densities, he could not show the optimality of his inequality for $v > 2\sqrt{E((X - \nu)^2)}/{\sqrt 3}.$
\end{remark}

\begin{remark}\label{symmetricGauss}
If $X - \nu$ is uniform on $(-c,c)$, it attains equality in (\ref{Gaussinequality}). Consequently, the bounds in Theorem \ref{th:Gauss} remain sharp for distributions symmetric about their mode.
\end{remark}

With the help of Khintchine's representation of unimodal distributions, \cite{Khin}, and Jensen's inequality a fifth proof of Gau\ss's inequality can be given.
To simplify notation we will formulate the essence of Gau\ss's inequality as follows.

\begin{theorem}[Gau\ss's inequality simplified]\label{th:Gausssimple}
For any unimodal random variable $X$ with mode at 0 and second moment equal to 1, the inequality
\begin{eqnarray}\label{Gausssimpleinequality}
\qquad P(|X| \geq v)\leq  {\left\{\begin{array}{rcl}
\frac 49 \, \frac 1{v^2},& & v \geq \frac 2{\sqrt 3}, \\
\\
1- \frac v{\sqrt3}, &  & 0\leq v \leq \frac 2{\sqrt 3},\\
\end{array}\right.}
\end{eqnarray}
holds with equality if the distribution of $X$ is the mixture of a uniform distribution on $(-c,c)$ with
\begin{equation}\label{uniformsimple}
c = \left(\tfrac 32 v \right) \vee {\sqrt 3}
\end{equation}
and a distribution degenerate at $0$ such that the point mass at $0$ equals
\begin{equation}\label{pointmasssimple}
\left(1- \tfrac 4{3v^2} \right) \vee 0.
\end{equation}
\end{theorem}
\begin{proof}[Proof of Theorem \ref{th:Gausssimple}]
By Khintchine's representation theorem $X$ can be written as $X=UY$ with $U$ and $Y$ independent random variables, $U$ uniformly distributed on the unit interval, and $E(Y^2)=3.$
As $y \mapsto 1-v/y$ is concave and increasing on $[v,\infty)$, Jensen's inequality yields
\begin{eqnarray}\label{Gs1}
\lefteqn{ P(|X| \geq v) = P( U|Y| \geq v) = E\left(1-\frac v{|Y|} \mid |Y| \geq v \right) P(|Y| \geq v) } \\
&& \leq \left(1- \frac v {E( |Y| \mid |Y| \geq v)} \right) P(|Y| \geq v)
\leq \left(1- \frac v {\sqrt{E( Y^2 \mid |Y| \geq v)}} \right) P(|Y| \geq v). \nonumber
\end{eqnarray}
Among the random variables $|Y|$ with $E(Y^2)=3$ that attain the maximum value of the right hand side of (\ref{Gs1}), there is a Bernoulli random variable
\begin{eqnarray}\label{Gs2}
\qquad  B = {\left\{\begin{array}{rccr}
a = \sqrt{E( Y^2 \mid |Y| < v)} && {\rm with~probability}\ \ 1-p \\
 \\
b = \sqrt{E( Y^2 \mid |Y| \geq v)} && {\rm with~probability}\ \ p\\
\end{array}\right.}
\end{eqnarray}
with $a^2(1-p) + b^2p =3$. As $p=(3-a^2)/(b^2-a^2)$ has to be positive and 1 at most, we have $0 \leq a < \sqrt{3} \leq b$.
Consequently, (\ref{Gs1}) and (\ref{Gs2}) yield
\begin{equation}\label{Gs3}
P(|X| \geq v) \leq \sup_{0 \leq a < \sqrt{3} \leq b} \left(1-\frac vb \right) \frac {3-a^2}{b^2-a^2}
= \sup_{\sqrt{3} \leq b} \left(1-\frac vb \right) \frac 3{b^2}.
\end{equation}
As the function $b\mapsto (1-v/b)b^{-2}$ is increasing-decreasing on the positive half line and attains its maximum at $3v/2$, the supremum at the right hand side of (\ref{Gs3}) is attained at $\sqrt 3$ or $3v/2$ depending on $v \leq 2/\sqrt{3}$ or $v \geq 2/\sqrt{3}$, respectively. Straightforward computations complete the proof of the Theorem.
\end{proof}

For a slightly broader class of unimodal distributions  we have our complete Markov version of the Gau\ss\ inequality, which extends (\ref{GaussM1}) as follows.

\begin{theorem}[Markov-Gau\ss \ inequality]\label{th:MarkovGauss}
For any unimodal random variable $X$ with mode at 0 and first absolute moment equal to 1, the inequality
\begin{eqnarray}\label{MarkovGaussinequality}
\qquad P(|X| \geq v)\leq  {\left\{\begin{array}{rcl}
\frac 1{2v},& & v \geq 1, \\
\\
1- \frac v2, &  & 0 \leq v \leq 1,\\
\end{array}\right.}
\end{eqnarray}
holds with equality if the distribution of $X$ is the mixture of a uniform distribution on $(-c,c)$ with $c = 2(v \vee 1)$
and a distribution degenerate at $0$ such that the point mass at $0$ equals $(1-2/c)\vee 0$.
\end{theorem}
\begin{proof}[Proof of Theorem \ref{th:MarkovGauss}]
Based on the first inequality from (\ref{Gs1}) we arrive at the analogue
\begin{equation}\label{Gs4}
P(|X| \geq v) \leq \sup_{0 \leq a < 2 \leq b} \left(1-\frac vb \right) \frac {2-a}{b-a}
= \sup_{2 \leq b} \left(1-\frac vb \right) \frac 2b = \min \left\{ 1-\frac v2, \frac 1{2v} \right\}
\end{equation}
of (\ref{Gs3}).
\end{proof}

\section{Gau\ss's inequality generalized}\label{Gig}

Our main interest in this paper is in sharp upper bounds on the probability that a symmetric unimodal random variable $X$ with mode at $\nu$ falls outside an arbitrary interval containing its mode, i.e. sharp upper bounds on
\begin{equation}\label{Gi1}
P(X -\nu \leq -u\ \  \mbox{\rm or}\ \  X -\nu \geq v), \quad u>0,\,\, v>0.
\end{equation}

We restrict attention to symmetric distributions as inequalities for (\ref{Gi1}) degenerate for possibly asymmetric distributions. This can be seen by noting that for $u\geq v$ the probability in (\ref{Gi1}) is bounded from above by (\ref{Gaussinequality}). These bounds of Gau\ss \,  are sharp for (\ref{Gi1}) as well, and are attained if the distribution of $X - \nu$ is the mixture of a uniform distribution on $(0,c)$ with $c$ as in (\ref{uniform}) and a distribution degenerate at $0$ such that the point mass at $0$ equals the value given in (\ref{pointmass}). In a similar way (\ref{Sellkeinequality}) yields a sharp upper bound to (\ref{Gi1}), when asymmetric unimodal distributions are allowed.

It is quite natural to restrict attention to symmetric unimodal distributions here, as equalities in the Gau\ss \, inequality (Theorem \ref{th:Gauss}) and in its generalization by the Sellke's (Theorem \ref{th:2Sellke}) are attained also by random variables $X - \nu$ that are uniform on $(-c,c)$ and $(-x_v,x_v),$ respectively. So, the upper bounds in Theorems \ref{th:2Sellke}, \ref{th:Gauss}, and \ref{th:MarkovGauss}  are still sharp if the class of unimodal distributions is restricted to symmetric distributions.

It should be noted that one-sided probabilities are allowed in (\ref{Gi1}), as we accept that either $u$ or $v$ could be equal to infinity.
In a recent paper \cite{Seme} a sharp upper bound on the probability (\ref{Gi1}) was given, thus generalizing Theorem \ref{th:Gauss} for symmetric unimodal distributions with finite second moment.
To keep notation simple we assume $\nu=0$ without loss of generality, and $E(X^2)=1$.
The parametrization chosen in \cite{Seme} is $h=(u+v)/2$, the half-length of the interval $(-u,v)$, and $m=(v-u)/2$, the midpoint of this interval.
We translate the inequality from \cite{Seme} into our notation as follows.

\begin{theorem}[Semenikhin]\label{th:symmetricunimodal}
Assume the distribution of the random variable $X$ is symmetric and unimodal with mode at 0 and has finite second moment equal to 1.
For $0< u \leq v$ one has
\begin{eqnarray}\label{MarkovGaussinequality2}
P(X\leq -u\ \  \mbox{\rm or}\ \  X \geq v) \leq
{\left\{\begin{array}{lcl}
1-\frac {u+v}{2\sqrt{3}}      &                         &  S_1 \\
\frac {16}{9(u+v)^2}          &                         &  S_2 \\
\frac 2{9u^2}                 & \textnormal{\it within} &  S_3 \\
\frac 12 \left(1-\frac{u^{2/3}}{(u+v)^{2/3}-u^{2/3}} + \frac 4{9\left((u+v)^{2/3}-u^{2/3} \right)^3} \right) & &  S_4 \\
\frac 12\left(1-\frac u{\sqrt{3}} \right) &             &  S_5
\end{array}\right.}
\end{eqnarray}
with
\begin{eqnarray}\label{setsS}
S_1 & = & \left\{(u,v)\,:\, 0<u \leq v \leq \psi(u) \wedge \left(\frac 4{\sqrt{3}} - u \right) \right\} \nonumber \\
S_2 & = & \left\{(u,v)\,:\, 0<u \leq v \leq (2\sqrt{2} -1)u,\, \frac 4{\sqrt{3}} - u \leq v \right\} \nonumber \\
S_3 & = & \left\{(u,v)\,:\, \frac 2{\sqrt{3}} \leq u,\, (2\sqrt{2} -1)u \leq v \right\} \\
S_4 & = & \left\{(u,v)\,:\, 0<u,\, \psi(u) \vee \left((2\sqrt{2} -1)u \right) \leq v \leq
\left( \left(1+\frac 2{\sqrt{3}\,u}\right)^{3/2} -1 \right) u \right\} \nonumber \\
S_5 & = & \left\{(u,v)\,:\, 0< u \leq \frac 2{\sqrt{3}},\, \left( \left(1+\frac 2{\sqrt{3}\,u}\right)^{3/2} -1 \right) u \leq v \right\}, \nonumber
\end{eqnarray}
where $\psi$ is defined by
\begin{equation}\label{psi}
\psi(u) = \frac 2{\sqrt{3}} -u + \sqrt{3} \left(\frac u3 \right)^{2/3} \left( \left[1 - \sqrt{1- \left(\frac u3 \right)^2} \right]^{1/3}
                            +\left(\frac u3 \right)^{2/3} \left[1 - \sqrt{1- \left(\frac u3 \right)^2} \right]^{-1/3} \right).
\end{equation}
Equalities hold within $S_1$ and $S_5$ if $X$ is uniformly distributed on the interval $[-\sqrt{3},\sqrt{3}]$.
Equality holds within $S_2$ if $X$ is uniform on $[-3(u+v)/4,\,3(u+v)/4]$ with probability $16/(3(u+v)^2)$ and has a point mass
at 0 with probability $1-16/(3(u+v)^2).$
Equality holds within $S_3$ if $X$ is uniform on $[-3u/2,\,3u/2]$ with probability $4/(3u^2)$ and has a point mass
at 0 with probability $1-4/(3u^2).$
Equality holds within $S_4$ if $X$ is uniform on $[-u_0,\, u_0]$ with probability $1-p$ and uniform on $[-u_1,\, u_1]$ with probability $p$, where
\begin{eqnarray}\label{puu}
\lefteqn{ u_0 = \frac 32 u^{1/3} \left((u+v)^{2/3} - u^{2/3} \right), \quad  u_1 = \frac 32 (u+v)^{1/3} \left((u+v)^{2/3} - u^{2/3} \right), \nonumber } \\
&& \hspace{5em} p = \frac {4-3 u^{2/3} \left((u+v)^{2/3} - u^{2/3} \right)^2}{3 \left((u+v)^{2/3} - u^{2/3} \right)^3} \hspace{5em}
\end{eqnarray}
hold.
\end{theorem}

\begin{remark}\label{quadrant}
Note that the union of the sets $S_1$--$S_5$ constitutes the part of the first quadrant that lies above the diagonal, namely
$\{(u,v)\,:\, 0<u \leq v\}$.
\end{remark}

\begin{remark}\label{largeuv}
Note that the diagonal $\{(u,v)\,:\, 0<u=v \}$ is contained in $S_1$ and $S_2$ and not in the other sets.
On this diagonal the corresponding bounds reduce to (\ref{Gauss1}); see also Theorem \ref{th:Gausssimple}.
\end{remark}

Finally, for the slightly broader class of unimodal distributions with finite first absolute moment we present our generalization to (\ref{Gi1}) of the Markovian Gau\ss\ inequality from Theorem \ref{th:MarkovGauss}.

\begin{theorem}\label{th:symmetricunimodalfirst}
Assume the distribution of the random variable $X$ is symmetric and unimodal with mode at 0 and has finite first absolute moment equal to 1.
For $0< u \leq v$ the probability $P(X\leq -u\ \  \mbox{\rm or}\ \  X \geq v)$ is bounded from above by
\begin{eqnarray}\label{MarkovGaussinequality2}
{\left\{\begin{array}{lcl}
1-\frac {u+v}4      &                         &  A_1 = \left\{(u,v)\,:\, 0<u \leq v \leq 2-u \right\} \\
\frac 12 - \frac u4 &                         &  A_2 = \left\{(u,v)\,:\, 0<u \leq 1,\, v > 2 + u^2/(2-u) \right\} \\
\frac 1{u+v}        & \textnormal{\it within} &  A_3 = \left\{(u,v)\,:\, 0<u \leq 1,\, 2-u < v \leq 2 + u^2/(2-u) \right\} \\
                    &                         & \hspace{4em} \cup \left\{(u,v)\,:\, 1 < u \leq v \leq 3u \right\} \\
\frac 1{4u}         &                         &  A_4 = \left\{(u,v)\,:\, 1<u,\, v > 3u \right\}.
\end{array}\right.}
\end{eqnarray}
Equalities hold within $A_1$ and $A_2$ if $X$ is uniformly distributed on the interval $[-2,2]$.
Equality holds within $A_3$ if $X$ is uniform on $[-(u+v),\, u+v]$ with probability $2/(u+v)$ and has a point mass
at 0 with probability $1-2/(u+v).$
Equality holds within $A_4$ if $X$ is uniform on $[-2u,\, 2u]$ with probability $1/u$ and has a point mass
at 0 with probability $1-1/u.$
\end{theorem}

\begin{remark}\label{quadrant}
Note that the union of the sets $A_1$--$A_4$ constitutes the part of the first quadrant on and above its diagonal, namely
$\{(u,v)\,:\, 0<u \leq v\}$.
\end{remark}

\noindent
\begin{proof}[Proof of Theorem \ref{th:symmetricunimodalfirst}]
By the Khintchine representation theorem (cf. \cite{Khin}) of unimodal random variables, $X$ may be written as $X=UY$ with $U$ and $Y$ independent random variables and $U$ uniformly distributed on $[0,1]$.
Since $X$ is symmetric, $Y$ has to be symmetric as well.
In view of $E|X|=1$ we have $E|Y|=2$.
By Jensen's inequality we obtain
\begin{eqnarray}\label{su5}
\lefteqn{P(X\leq -u\ \  \mbox{\rm or}\ \  X \geq v) = P(X\geq u) + P(X \geq v) } \\
& = & E\left(1- \frac uY\,\Big|\,u \leq Y < v \right)\,P(u \leq Y < v) + E\left(2- \frac {u+v}Y\,\Big|\,Y \geq v \right)\,P(Y\geq v) \nonumber \\
& \leq &  \left[1- \frac u{E(Y \mid u \leq Y < v)}\right] \, P(u \leq Y < v) + \left[2- \frac {u+v}{E(Y \mid Y\geq v)}\right] \, P(Y\geq v). \nonumber
\end{eqnarray}
With the notation $a = E(Y \mid u \leq Y < v)>0,\, b = E(Y \mid Y\geq v)>0,\, p = P(u \leq Y < v),$ and $q = P(Y\geq v)$ this implies
\begin{eqnarray}\label{su6}
\lefteqn{ P(X\leq -u\ \  \mbox{\rm or}\ \  X \geq v) \leq \sup \left\{ \left(1-\frac ua \right)\,p + \left(2-\frac {u+v}b \right)\,q \ :\ 0 < u \leq a \leq v \leq b, \right. \nonumber } \\
&& \hspace{15em} \left. p \geq 0,\ q \geq 0,\ p+q \leq \frac 12,\ ap + bq \leq 1 \right\}.
\end{eqnarray}
By increasing $b$ if necessary, we see that this supremum is attained at $ap + bq =1.$
Then, $q= (1-ap)/b$ holds and $0 \leq q$ and $p+q \leq 1/2$ imply
\begin{equation}\label{su7}
0 \leq p \leq \frac 1a \wedge \frac{b-2}{2(b-a)}\ \ {\rm and\ hence}\ \ b \geq 2.
\end{equation}
Observe that $1/a \leq (b-2)/(2(b-a))$ is equivalent to $a \geq 2$.
So, by linearity in $p$ we see that the right hand side of (\ref{su6}) equals
\begin{eqnarray}\label{su8}
\lefteqn{ \sup \left\{ \left( 1-\frac ua \right) p + \left( 2- \frac{u+v}b \right) \frac {1-ap}b \,:\, 0 < u \leq a \leq v \leq v \vee 2 \leq b, \right. \nonumber } \\
&& \hspace{5cm} \left. 0 \leq p \leq \frac 1a \wedge \frac{b-2}{2(b-a)} \right\} \nonumber \\
&& = \sup_{0 < u \leq a \leq v \leq v \vee 2 \leq b} \max \left\{ \left( 2- \frac{u+v}b \right) \frac 1b, \
    \left( 1-\frac ua \right) \frac 1a {\bf 1}_{[a \geq 2]}, \right. \\
&& \hspace{5em} \left. \left[ \left( 1-\frac ua \right) \frac{b-2}{2(b-a)}
    + \left( 2- \frac{u+v}b \right) \frac{2-a}{2(b-a)} \right] {\bf 1}_{[a \leq 2]} \right\}. \nonumber
\end{eqnarray}
Together with (\ref{su6}) and (\ref{su7}) this implies (with $0 < u \leq v$)
\begin{eqnarray}\label{su9}
\lefteqn{ P(X\leq -u\ \  \mbox{\rm or}\ \  X \geq v)
        \leq \max \left\{ \sup _{v \vee 2\leq b} \left( 2- \frac{u+v}b \right) \frac 1b, \
                     \sup _{u \vee 2 \leq a \leq v} \left( 1-\frac ua \right) \frac 1a, \right. \nonumber } \\
&& \hspace{4em} \left. \sup_{u \leq a \leq v \wedge 2 \leq v \vee 2 \leq b} \left( 1-\frac ua \right) \frac{b-2}{2(b-a)}
    + \left( 2- \frac{u+v}b \right) \frac{2-a}{2(b-a)} \right\}.
\end{eqnarray}
As $b \mapsto (2-(u+v)/b)/b$ is increasing-decreasing and attains its maximum at $b=u+v$, we obtain
\begin{equation}\label{su10}
\sup _{v \vee 2 \leq b} \left( 2- \frac{u+v}b \right) \frac 1b
= \frac 1{u+v} {\bf 1}_{[0<u \leq v, u+v \geq 2]} + \left(1 - \frac {u+v}4 \right) {\bf 1}_{[0<u \leq v, u+v < 2]}.
\end{equation}
As $a \mapsto (1-u/a)/a$ is increasing-decreasing and attains its maximum at $a=2u$, we obtain
\begin{equation}\label{su11}
\sup _{u \vee 2 \leq a \leq v} \left( 1-\frac ua \right) \frac 1a
             = \frac 1{4u} {\bf 1}_{[1 \leq u \leq 2u \leq v]} + \left( 1-\frac uv \right) \frac 1v {\bf 1}_{[0 < u \leq v, 2 \leq v \leq 2u]}.
\end{equation}

To study under $0<u \leq v \wedge 2$
\begin{equation}\label{su12}
\sup_{u \leq a \leq v \wedge 2 \leq v \vee 2 \leq b} \left( 1-\frac ua \right) \frac{b-2}{2(b-a)}
+ \left( 2- \frac{u+v}b \right) \frac{2-a}{2(b-a)}
\end{equation}
we differentiate with respect to $b$ and note that the derivative is nonnegative if and only if
$(a+u)b^2 -2a(u+v)b +(u+v)a^2 \leq 0$ holds, i.e., if and only if
\begin{equation}\label{su13}
b \in \left[ \frac a{a+u} \left( u+v - \sqrt{(u+v)(v-a)} \right), \frac a{a+u} \left( u+v + \sqrt{(u+v)(v-a)} \right) \right]
\end{equation}
holds.
Observe that $v$ belongs to this interval and that hence the function in (\ref{su12}) is increasing-decreasing on $[v, \infty)$ with its maximum attained at the right endpoint from (\ref{su13}).
We discern two cases.
\begin{itemize}

\item[\bf A] $0<u \leq a \leq 2 \leq v \leq b$ \\
Substituting in this situation the value of the right endpoint from (\ref{su13}) into the function from (\ref{su12}) we arrive at
\begin{equation}\label{su14}
\sup_{[0<u \leq a \leq 2 \leq v]} 1- \frac {u+a}{a^2} +  \frac {2-a}{a^2} \left( \sqrt{u+v} - \sqrt{v-a} \right) \sqrt{u+v}.
\end{equation}
The derivative of this function with respect to $a$ is nonnegative if and only if
\begin{equation}\label{su15}
\left( 2(u+v+1)a + 8v - 4u \right)\sqrt{v-a} + \left( 2(4-a)(v-a) +a(2-a) \right) \sqrt{u+v} \geq 0
\end{equation}
holds.
Since this is clearly the case in the present situation the supremum in (\ref{su14}) is attained at $a=2$ yielding the value $(1-u/2)/2$.

\item[\bf B] $0<u \leq a \leq v \leq 2 \leq b$ \\
If in this situation 2 belongs to the interval from (\ref{su13}) the supremum over $b$ of the function from (\ref{su12}) is attained at the value of the right endpoint of the interval from (\ref{su13}).
Then like in (\ref{su14}) and with the help of (\ref{su15}) we arrive at
\begin{equation}\label{su16}
\sup_{[0<u \leq a \leq v \leq 2]} 1- \frac {u+a}{a^2} +  \frac {2-a}{a^2} \left( \sqrt{u+v} - \sqrt{v-a} \right) \sqrt{u+v}
\leq \frac 12 - \frac u4.
\end{equation}
If 2 does not belong to the interval from (\ref{su13}) the supremum over $b$ of the function from (\ref{su12}) is attained at $b=2$, which yields the value $1-(u+v)/4$.
Note that we have $(1-u/2)/2 \leq 1-(u+v)/4$ here in view of $v \leq 2$.
\end{itemize}
Combining {\bf A} and {\bf B} we see that the supremum in (\ref{su12}) equals
\begin{equation}\label{su17}
\left( \frac 12 - \frac u4 \right) {\bf 1}_{[0<u \leq 2 \leq v]} + \left( 1 - \frac{u+v}4 \right) {\bf 1}_{[0<u \leq v < 2]}.
\end{equation}
Combining (\ref{su9}). (\ref{su10}), (\ref{su11}) and (\ref{su17}) we obtain by straightforward, but tedious computations
\begin{eqnarray}\label{su18}
\lefteqn{ P(X\leq -u\ \  \mbox{\rm or}\ \  X \geq v) \nonumber } \\
&& \leq \max \left\{ {\bf 1}_{[0<u \leq v,\, v \leq 2-u]} \left(1 - \frac{u+v}4 \right),\, {\bf 1}_{[0<u \leq v \leq 2,\, v> 2-u]} \frac 1{u+v}, \right. \nonumber \\
&& \hspace{5em} {\bf 1}_{[0<u \leq 1,\, 2 < v \leq 2+ u^2/(2-u)]} \left(\frac 12 - \frac u4 \right),\, {\bf 1}_{[0<u \leq 1,\, v > 2+ u^2/(2-u)]} \frac1{u+v}, \nonumber \\
&& \hspace{5em} {\bf 1}_{[1<u \leq 2,\, 2u < v \leq 3u]} \frac 1{u+v},\, {\bf 1}_{[1<u \leq 2,\, v > 3u]} \frac 1{4u}, \\
&& \hspace{5em} {\bf 1}_{[1<u \leq 2,\, 2 < v \leq 2u]} \frac 1{u+v},\, {\bf 1}_{[2 < u \leq v \leq 2u]} \frac 1{u+v}, \nonumber \\
&& \hspace{5em} \left. {\bf 1}_{[2 < u,\, 2u <v \leq 3u]} \frac 1{u+v},\, {\bf 1}_{[2 < u, v > 3u]} \frac 1{4u} \right\} \nonumber \\
&& = \left(1 - \frac{u+v}4 \right) {\bf 1}_{[0<u \leq v \leq 2-u]} + \frac 1{u+v} {\bf 1}_{[0<u \leq 1,\, 2-u < v \leq 2 + u^2/(2-u)]} \nonumber \\
&& \hspace{3em} + \frac 1{u+v} {\bf 1}_{[1 < u \leq v \leq 3u]} + \left( \frac 12 - \frac u4 \right) {\bf 1}_{[0<u \leq 1,\, v >2 + u^2/(2-u)]}
    + \frac 1{4u} {\bf 1}_{[1<u,\, v > 3u]}. \nonumber
\end{eqnarray}
This inequality may be reformulated as in (\ref{MarkovGaussinequality2}).
Simple computations show that the random variables $X$ defined in the theorem attain equalities in (\ref{MarkovGaussinequality2}).
\end{proof}

\bibliographystyle{amsplain}

\end{document}